\documentclass[11pt]{article}
\usepackage{mathtools}
\usepackage{ctable}
\usepackage{amsmath}
\usepackage{amsfonts}
\usepackage{latexsym}
\usepackage{amsthm}
\usepackage{color}

\newtheorem{theorem}{Theorem}[section]
\newtheorem{corollary}{Corollary}[section]
\newtheorem{lemma}{Lemma}[section]
\newtheorem{proposition}{Proposition}[section]

\theoremstyle{definition}
\newtheorem{definition}{Definition}
\newtheorem{remark}{Remark}
\newtheorem{example}{Example}
\title{\huge \bf On Generalized Reversed Aging\\ Intensity Functions\\ 
}
\author{
{\sc Francesco Buono}\\
 Universit\`a di Napoli Federico II\\
 Italy\\
\and {\sc Maria Longobardi}\\
Universit\`a di Napoli Federico II\\
 Italy\\
\and {\sc Magdalena Szymkowiak}\\
Poznan University of Technology\\
 Poland\\
}
\begin{document}
\date{\today}
\maketitle
\begin{abstract}
The reversed aging intensity function is defined as the ratio of the instantaneous reversed hazard rate to the baseline value of the reversed hazard rate. It analyzes the aging property quantitatively, the higher the reversed aging intensity, the weaker the tendency of aging. In this paper, a family of generalized reversed aging intensity functions is introduced and studied. Those functions depend on a real parameter. If the parameter is positive they characterize uniquely the distribution functions of univariate positive absolutely continuous random variables, in the opposite case they characterize families of distributions.  Furthermore, the generalized reversed aging intensity orders are defined and studied. Finally, several numerical examples are given.
\end{abstract}

\medskip\noindent
{\em Keywords:\/} 
Generalized reversed aging intensity, Reversed hazard rate, Generalized Pareto distribution, Generalized reversed aging intensity order.

\medskip\noindent
{\em AMS Subject Classification}: 60E15, 60E20, 62N05

\section{Introduction}

Let $X$ be a non--negative and absolutely continuous random variable with cumulative distribution function (cdf) $F$, probability density function (pdf) $f$ and survival function (sf) $\overline F$. In reliability theory $F$ is also known as unreliability function whereas $\overline F$ as reliability function. In this context a great importance has the hazard rate function $r$ of $X$, also known as the force of mortality or the failure rate,  where $X$ is the survival model of a life or a system being studied.  This definition will cover discrete survival models as well as mixed survival models. In the same way we define the reversed hazard rate $\breve{r}$ of $X$, that has attracted the attention of researchers. In a certain sense it is the dual function of the hazard rate and it bears some interesting features useful in reliability analysis (see also Block and Savits 1998; Finkelstein 2002). 

Let $X$ be a random variable with sf $\overline F$ and cdf $F$. We define, for $x$ such that $\overline F(x)>0$, the hazard rate function of $X$ at $x$, $r(x)$, in the following way:
\begin{align*}
r(x) &= \lim_{\Delta x\to 0^+}\frac{\mathbb P(x<X\leq x+\Delta x|X>x)}{\Delta x}\\
&= \frac{1}{\overline F(x)}\lim_{\Delta x\to 0^+}\frac{\mathbb P(x<X\leq x+\Delta x)}{\Delta x}.
\end{align*}
Moreover we define, for $x$ such that $F(x)>0$, the reversed hazard rate function of $X$ at $x$, $\breve{r}(x)$ (see Bartoszewicz 2009 for the notation), in the following way:
\begin{align*}
\breve{r}(x)&=\lim_{\Delta x\to 0^+}\frac{\mathbb P(x-\Delta x<X\leq x|X\leq x)}{\Delta x}\\
&=\frac{1}{F(x)}\lim_{\Delta x\to 0^+}\frac{\mathbb P(x-\Delta x<X\leq x)}{\Delta x}.
\end{align*}
The reversed hazard rate $\breve r(x)$ can be treated as the instantaneous failure rate occurring immediately before the time point $x$ (the failure occurs just before the time point $x$, given that the unit has not survived longer than time $x$).

So, if $X$ is an absolutely continuous random variable with density $f$, for $x$ such that $\overline F(x)>0$, the hazard rate function is 
\begin{equation}
r(x)=\frac{1}{\overline F(x)}\lim_{\Delta x\to 0^+}\frac{\mathbb P(x<X\leq x+\Delta x)}{\Delta x}=\frac{f(x)}{\overline F(x)},
\end{equation}
while, for $x$ such that $F(x)>0$, the reversed hazard rate function is 
\begin{equation}
\breve r(x)=\frac{1}{F(x)}\lim_{\Delta x\to 0^+}\frac{\mathbb P(x-\Delta x<X\leq x)}{\Delta x}=\frac{f(x)}{F(x)}.
\end{equation}

By the hazard rate function we introduce the aging intensity function $L$ that is defined for $x>0$ as
\begin{equation}
L(x)=\frac{-x f(x)}{\overline{F}(x) \log \overline{F}(x)}=\frac{-x r(x)}{\log \overline{F}(x)},
\end{equation}
where $\log$ denotes the natural logarithm.
It can be showed that the survival function of an absolutely continuous random variable and its aging intensity function are related by a relationship and that under some conditions a function determines a family of survival functions and it is their aging intensity function, for more details see Szymkowiak (2018a).

The reversed aging intensity function $\breve L(x)$ is defined, for $x>0$, as follows (see also Rezaei and Khalef 2014)
\begin{equation}
\label{eq3}
\breve L(x)=\frac{-xf(x)}{F(x)\log F(x)}=\frac{-x\breve r(x)}{\log F(x)}.
\end{equation}

The reversed aging intensity function can be expressed also in a different way by observing that the cumulative reversed hazard rate function defined as
$$\breve R(x)=\int_x^{+\infty} \breve r(t) \ \mathrm dt = \log F(t)\big| _{t=x}^{t\rightarrow+\infty}=-\log F(x),$$
can be treated as the total amount of failures accumulated after the time point $x$. So $\breve H(x)=\frac{1}{x}\breve R(x)$, being the proportion between the total amount of failures accumulated after the time point $x$ and the time $x$ for which the unit is still survived, can be considered as the baseline value of the reversed hazard rate. Then, \eqref{eq3} can be written as $$\breve L(x)=\frac{x\breve r(x)}{\breve R(x)}=\frac{\breve r(x)}{\breve H(x)},$$
and so the reversed aging intensity function, defined as the ratio of the instantaneous reversed hazard rate $\breve{r}$ to the baseline value of the reversed hazard rate $\breve{H}$, expresses the units average aging behavior: the higher the reversed aging intensity (it means the higher the instantaneous reversed hazard rate, and the smaller the total amount of failures accumulated after the time point $x$, and the higher the the time $x$ for which the unit is still survived),  the weaker the tendency of aging.

It is the analogous for the future of the aging intensity function, introduced and studied by Bhattacharjee, Nanda and Misra (2013), Jiang, Ji and Xiao (2003), Nanda, Bhattacharjee and Alam (2007) and Szymkowiak (2018a). The concept of aging intensity function was generalized by Szymkowiak (2018b). In Section 2, we define the generalized reversed aging intensity functions and in the particular case in which the random variable has a generalized Pareto distribution function we generalize our results and study monotonicity properties. In Section~3, we give some characterizations with use of our new aging intensities. Some examples of characterization are given in Section 4. In Section 5, we study the family of new stochastic orders called $\alpha$--generalized reversed aging intensity orders. Then, in Section 6 we present examples of analysis of $\alpha$--generalized reversed aging intensity through generated and real data.

\section{Generalized reversed aging intensity functions}

Let, for $x>0$, $W_0(x)=1-\exp(-x)$, i.e., $W_0$ is the distribution function of an exponential variable with parameter 1, so $\breve R(x)=W_0^{-1}(1-F(x))$. In fact, $W_0^{-1}(x)=-\log(1-x)$ and so
\begin{equation*}W_0^{-1}(1-F(x))=-\log F(x)=\breve R(x).
\end{equation*}
Replacing $W_0$ with a strictly increasing distribution function $G$ with density $g$, it is possible to generalize the concepts of reversed hazard rate function, cumulative reversed hazard rate function and reversed aging intensity function. The generalization of the hazard rate function was introduced by Barlow and Zwet (1969a, 1969b).

\begin{definition}
Let $X$ be a non--negative and absolutely continuous random variable with cdf $F$. Let $G$ be a strictly increasing distribution function with density $g$. We define the $G$--generalized cumulative reversed hazard rate function, $\breve R_{G}$, the $G$--generalized reversed hazard rate function, $\breve r_{G}$, the $G$--generalized reversed aging intensity function, $\breve L_{G}$, of $X$ as
\begin{align}
&\breve R_{G}(x)=G^{-1}(1-F(x)),\\
&\breve r_{G}(x)=-\frac{\mathrm d \breve R_{G}(x)}{\mathrm dx}=\frac{f(x)}{g(G^{-1}(1-F(x)))},\\
&\breve L_{G}(x)=\frac{x\breve r_{G}(x)}{\breve R_{G}(x)}=\frac{xf(x)}{g(G^{-1}(1-F(x))) G^{-1}(1-F(x))}.
\end{align}
\end{definition}
A very interesting case, because it provides intuitive results, is the one in which the distribution function $G$ is the distribution function of a generalized Pareto distribution.

\begin{definition}
A random variable $X_\alpha$ follows a generalized Pareto distribution with parameter $\alpha\in\mathbb R$ if the distribution function $W_\alpha$ is expressed as (see Pickands 1975):
$$W_\alpha(x)=\begin{cases} 1-(1-\alpha x)^{\frac{1}{\alpha}}, & \mbox{for }\begin{cases} x>0, & \mbox{if }\alpha<0 \\ 0<x<\frac{1}{\alpha}, & \mbox{if } \alpha>0 
\end{cases} \\ 1-\exp(-x), & \mbox{for } x>0  \mbox{ if } \alpha=0
\end{cases}$$
\end{definition}

\begin{remark}
For $\alpha=0$ we have the distribution function of an exponential variable with parameter $1$.
\end{remark}

From the distribution function it is possible to obtain the quantile and the density function. In particular we have
$$W_{\alpha}^{-1}(x)=\begin{cases} \frac{1}{\alpha}[1-(1-x)^{\alpha}], & \mbox{for } 0<x<1, \mbox{ if }\alpha\ne0 \\ -\log(1-x), & \mbox{for } 0<x<1,  \mbox{ if } \alpha=0
\end{cases}$$
$$w_{\alpha}(x)=\begin{cases} (1-\alpha x)^{\frac{1-\alpha}{\alpha}}, & \mbox{for }\begin{cases} x>0, & \mbox{if }\alpha<0 \\ 0<x<\frac{1}{\alpha}, & \mbox{if } \alpha>0 
\end{cases} \\ \exp(-x), & \mbox{for } x>0,  \mbox{ if } \alpha=0
\end{cases}$$

Let $X$ be a non--negative and absolutely continuous random variable with cdf $F$ and pdf $f$. Then it is possible to determine the $W_\alpha$-~generalized cumulative reversed hazard rate function and the $W_\alpha$-~generalized reversed hazard rate function in the following way:
$$\breve R_{W_\alpha}(x)=W^{-1}_{\alpha} (1-F(x))=\begin{cases} \frac{1}{\alpha}[1-F^{\alpha}(x)], & \mbox{for } x>0, \mbox{ if }\alpha\ne0 \\ -\log F(x), & \mbox{for } x>0,  \mbox{ if } \alpha=0
\end{cases}$$
$$\breve r_{W_\alpha}(x)=-\frac{\mathrm d \breve{R}_{\alpha}(x)}{\mathrm dx}=F^{\alpha-1}(x)f(x), \mbox{ for } x>0.$$
For the sake of simplicity, those functions can be, respectively, indicated by $\breve R_{\alpha}$, $\breve r_{\alpha}$ and we can refer to them as the $\alpha$--generalized cumulative reversed hazard rate function and the $\alpha$--generalized reversed hazard rate function.

\begin{remark}
The $1$--generalized reversed hazard rate function is equal to the density function. In fact the density function gives a first rough illustration of the aging tendency of the random variable by its monotonicity. The $0$--generalized reversed hazard rate function is equal to the usual reversed hazard rate function.
\end{remark}

From these functions, it is possible to introduce the $\alpha$--generalized reversed aging intensity function
\begin{equation}
\label{eq4}
\breve L_{\alpha}(x)=\frac{\breve{r}_{\alpha}(x)}{\frac{1}{x}\breve{R}_{\alpha}(x)}=\begin{cases} \frac{\alpha xF^{\alpha-1}(x)f(x)}{1-F^{\alpha}(x)}, & \mbox{for } x>0, \mbox{ if }\alpha\ne0 \\ \frac{-xf(x)}{F(x)\log F(x)}, & \mbox{for } x>0,  \mbox{ if } \alpha=0
\end{cases}
\end{equation}
The $\alpha$--generalized reversed aging intensity function describes the relationship between the instantaneous value of the $\alpha$--generalized reversed hazard rate function $\breve{r}_{\alpha}(x)$ and the baseline value of the $\alpha$--generalized reversed hazard rate function $\frac{1}{x}\breve{R}_{\alpha}(x)$. The higher the $\alpha$--generalized reversed aging intensity function (it means the higher the actual value of the $\alpha$--generalized reversed hazard rate function respect to its baseline value), the weaker the tendency of aging. Moreover, the $\alpha$--generalized reversed aging intensity function can be treated as the elasticity (see Sydsaeter and Hammond 2012), except for the sign, of the $\alpha$--generalized cumulative reversed hazard rate function, i.e., it indicates how much the function $\breve R_{\alpha}$ changes if $x$ changes by a small amount.

We recall the definition of $\alpha$--generalized aging intensity functions, $L_{\alpha}$. These functions are defined by Szymkowiak (2018b) in the following way
\begin{equation}
L_{\alpha}(x)=\begin{cases} \frac{\alpha x(1-F(x))^{\alpha-1}f(x)}{1-(1-F(x))^{\alpha}}, & \mbox{for } x>0, \mbox{ if }\alpha\ne0 \\ \frac{-xf(x)}{(1-F(x))\log(1-F(x))}, & \mbox{for } x>0,  \mbox{ if } \alpha=0
\end{cases}
\end{equation}

\begin{remark}
\label{rem}
The $0$--generalized reversed aging intensity function is equal to the usual reversed aging intensity function. If $\alpha=1$ we have
$$\breve L_{1}(x)=\frac{xf(x)}{\overline F(x)},$$
i.e., it is the negative of the elasticity of the survival function $\overline F$, they are equal in modulus. 

If $\alpha=n\in\mathbb N$ we have
$$\breve L_{n}(x)=\frac{n xF^{n-1}(x)f(x)}{1-F^{n}(x)},$$
where the denominator is the survival function of the largest order statistic for a sample of $n$ i.i.d. variables, while the numerator is composed by $x$ multiplied for the density of this order statistic. So $\breve L_{n}$ can be considered as the negative of the elasticity for the survival function of the largest order statistic. 

If $\alpha=-1$ we have
$$\breve L_{-1}(x)=\frac{-x(F(x))^{-2}f(x)}{1-(F(x))^{-1}}=\frac{-xf(x)}{F^2(x)\frac{F(x)-1}{F(x)}}=\frac{xf(x)}{F(x)(1-F(x))},$$
and so
$$\breve L_{-1}(x)=x LOR_X(x)=L_{-1}(x),$$
where $LOR_X$ is the log-odds rate of $X$ (see Zimmer, Wang and Pathak 1998).
\end{remark}

The next proposition analyzes the monotonicity of $\alpha$--generalized reversed aging intensity functions respect to the parameter $\alpha$. This result could be important if we introduce stochastic orders based on $\alpha$--generalized reversed aging intensity functions, i.e., $\alpha RAI$ orders, and compare these orders as $\alpha$ varies (see Section \ref{S5}).

\begin{proposition}
Let $X$ be a non--negative and absolutely continuous random variable with cdf $F$ and pdf $f$. Then the $\alpha$--generalized reversed aging intensity function is decreasing respect to $\alpha\in\mathbb R$, $\forall x\in(0,+\infty)$.
\end{proposition}

\begin{proof}
For some $c\in(0,1)$ we consider the function $h_c(\alpha)=\frac{\alpha c^\alpha}{1-c^\alpha}$, for $\alpha\ne0$. Then
$$\frac{\mathrm dh_c(\alpha)}{\mathrm d\alpha}=\frac{c^\alpha(1-c^\alpha+\log c^\alpha)}{(1-c^\alpha)^2}.$$
That derivative is negative because $c^\alpha\in(0,+\infty)$ and the function $k(t)=1-t+\log t$ is negative for $t>0$ and different from 1. In fact, $k(1)=0$ and 1 is maximum point for this function. So $h_c$ is decreasing in $(-\infty,0)\cup(0,+\infty)$. Defining the extension for continuity in 0 of $h_c$,
$$h_c(0)=\lim_{\alpha\to 0}h_c(\alpha)=\lim_{\alpha\to 0}\frac{\alpha c^\alpha}{1-c^\alpha}=\lim_{\alpha\to 0}\frac{c^\alpha+\alpha c^\alpha\log c}{-c^\alpha\log c}=-\frac{1}{\log c},$$
it is possible to say that $h_c$ is decreasing in $\mathbb R$.

Fixing $c=F(x)$, with $x>0$, and multiplying $h_{F(x)}(\alpha)$ for the positive factor $\frac{xf(x)}{F(x)}$ we get that the following function 
$$\frac{xf(x)}{F(x)}h_{F(x)}(\alpha)=\begin{cases} \frac{\alpha xF^{\alpha-1}(x)f(x)}{1-F^{\alpha}(x)}, &\mbox{ if }\alpha\ne0 \\ \frac{-xf(x)}{F(x)\log F(x)}, &  \mbox{ if } \alpha=0
\end{cases}\mbox{  }=\breve L_{\alpha}(x)$$
is decreasing in $\alpha$ as $x$ is fixed.
\end{proof}

\section{Characterizations with use of $\alpha$--generalized reversed aging intensity}

In reliability theory some functions characterize the associated distribution function. For example it was showed in Barlow and Proschan (1996) that the hazard rate of an absolutely continuous random variable uniquely determines its distribution function.

In the following theorem we show that, for $\alpha<0$, the distribution function of a non--negative and absolutely continuous random variable is defined by the $\alpha$--generalized reversed aging intensity function and that, under some conditions, a function can be considerated as the $\alpha$--generalized reversed aging intensity function for a family of random variables.

\begin{theorem}
\label{T20}
Let $X$ be a non--negative and absolutely continuous random variable with cdf $F$ and let $\breve L_{\alpha}$ be its $\alpha$--generalized reversed aging intensity function with $\alpha<0$. Then $F$ and $\breve L_{\alpha}$ are related, for all $a\in(0,+\infty)$, by the relationship
\begin{equation}
\label{eq1}
F(x)=\left[1-(1-F^{\alpha}(a))\exp\left(-\int_a^x \frac{\breve L_{\alpha}(t)}{t}\ \mathrm dt \right)\right]^{\frac{1}{\alpha}},\ \ \  x\in(0,+\infty).
\end{equation}
Moreover, a function $\breve L$ defined on $(0,+\infty)$ and satisfying, for $a\in(0,+\infty)$, the following conditions:
\begin{itemize}
\item[(1)] $0\leq \breve L(x)<+\infty,$ for all $x \in (0,+\infty)$;
\item[(2)] $\lim_{x \to 0^+} \int_x^a \frac{\breve L(t)}{t}\ \mathrm dt = +\infty$;
\item[(3)] $\lim_{x \to +\infty} \int_a^x \frac{\breve L(t)}{t}\ \mathrm dt = +\infty$;
\end{itemize}
determines, for $\alpha<0$, a family of absolutely continuous distribution functions $F_k$ by the relationship
\begin{eqnarray}
\nonumber
F_k(x) &=& 1-W_{\alpha}\left(k\exp\left(-\int_a^x \frac{\breve L(t)}{t}\ \mathrm dt \right)\right) \\
													 &=& \left[1-k\alpha\exp\left(-\int_a^x \frac{\breve L(t)}{t}\ \mathrm dt \right)\right]^{\frac{1}{\alpha}}, \ \ \ x\in(0,+\infty), 
\end{eqnarray}
for varying the parameter $k\in(0,+\infty)$ and it is the $\alpha$--generalized reversed aging intensity function for those distribution functions.
\end{theorem}

\begin{proof}
Fix distribution function $F$ with respective density function $f$, and put $\alpha<0$. From the definition of $\breve L_{\alpha}$ it is possible to obtain
$$\frac{\breve L_{\alpha}(t)}{t}=\frac{\alpha F^{\alpha-1}(t)f(t)}{1-F^{\alpha}(t)}, \mbox{  } t\in(0,+\infty).$$
By integrating both members between $a$ and $x$, for an arbitrary $a\in(0,+\infty)$, we get
\begin{eqnarray*}
\int_a^x \frac{\breve L_{\alpha}(t)}{t}\ \mathrm dt &=& \int_a^x \frac{\alpha F^{\alpha-1}(t)f(t)}{1-F^{\alpha}(t)}\ \mathrm dt \\
													 &=& -\log\frac{1-F^{\alpha}(x)}{1-F^{\alpha}(a)},
\end{eqnarray*}
therefore
$$1-F^{\alpha}(x)=(1-F^{\alpha}(a))\exp\left(-\int_a^x \frac{\breve L_{\alpha}(t)}{t}\ \mathrm dt\right),$$
and so we get \eqref{eq1}.

Let $\breve L$ be a function defined on $(0,+\infty)$ and satisfying, for $a\in(0,+\infty)$, the conditions $\textit{(1)}$, $\textit{(2)}$, $\textit{(3)}$.
We show that $1- W_{\alpha}\left(k\exp\left(-\int_a^x \frac{\breve L(t)}{t}\ \mathrm dt \right)\right)= F_k(x)$ defines a distribution function of a non--negative and absolutely continuous random variable, with $k\in(0,+\infty)$.

In fact, from $\textit{(2)}$ it follows that $\lim_{x \to 0^+} \int_a^x \frac{\breve L(t)}{t}\ \mathrm dt = -\infty$ and so we conclude that $\lim_{x \to 0^+}  F_k(x)=0$, while from $\textit{(3)}$ we have $\lim_{x \to +\infty} F_k(x)=1$. 

Since $W_\alpha$ is increasing, $k>0$ and the exponential function is increasing, in order to show that it is an increasing function we have to prove that $-\int_a^x \frac{\breve L(t)}{t}\ \mathrm dt$ is a decreasing function in $x$ i.e., $\int_a^x \frac{\breve L(t)}{t}\ \mathrm dt$ is increasing in $x$. From $\textit{(1)}$ and from the assumptions about the interval in which we have $a$ and $x$ it follows that the integrand is non negative. Let $x_1, x_2$ be such that $0<x_1<x_2<+\infty$. If $x_1\ge a$ then we have two non negative quantities and $\int_a^{x_1} \frac{\breve L(t)}{t}\ \mathrm dt\leq\int_a^{x_2} \frac{\breve L(t)}{t}\ \mathrm dt$ because $(a,x_1)\subset(a,x_2)$. If $x_1\leq a < x_2$, then $\int_a^{x_1} \frac{\breve L(t)}{t}\ \mathrm dt\leq0\leq\int_a^{x_2} \frac{\breve L(t)}{t}\ \mathrm dt$. If, finally, $x_1 < x_2\leq a$, we have two non positive quantities and, for a reasoning similar to the first case, $\int_{x_1}^a \frac{\breve L(t)}{t}\ \mathrm dt\ge\int_{x_2}^a \frac{\breve L(t)}{t}\ \mathrm dt$ and so $\int_a^{x_1} \frac{\breve L(t)}{t}\ \mathrm dt\leq\int_a^{x_2} \frac{\breve L(t)}{t}\ \mathrm dt$.

Since $W_\alpha$, the exponential function, the multiplication for a scalar and the indefinite integral $x\mapsto \int_a^x \frac{\breve L(t)}{t}\ \mathrm dt$ with respect to the Lebesgue measure are continuous functions, we have a continuous function. In order to obtain the absolute continuity of $F_k$, it suffices to observe that the derivative
\begin{equation*}
F'_k(x)=\left[1-k\alpha\exp\left(-\int_a^x \frac{\breve L(t)}{t} \mathrm dt\right)\right]^{\frac{1}{\alpha}-1} (-k)\exp\left(-\int_a^x \frac{\breve L(t)}{t} \mathrm dt\right)\left(-\frac{\breve L(x)}{x} \right)
\end{equation*}
is non--negative in $x>0$.

To show that $\breve L$ is $\alpha$--generalized reversed aging intensity function related to those distribution functions we have to observe that $F_k(a)=[1-k\alpha]^{\frac{1}{\alpha}}$, and so $k\alpha=1-F_k^\alpha(a)$ i.e., $F_k$ and $\breve L$ are related by the relationship expressed in the first part of the theorem.
\end{proof}

\begin{remark}
The expression $W_{\alpha}\left(k\exp\left(-\int_a^x \frac{\breve L(t)}{t}\ \mathrm dt \right)\right)$ depends only on the parameter $k\in(0,+\infty)$ because the dependence from $a\in(0,+\infty)$ is fictitious. In fact, replacing $a$ with $b\in(0,+\infty)$ we get
\begin{align*}
&W_{\alpha}\left(k\exp\left(-\int_b^x \frac{\breve L(t)}{t}\ \mathrm dt \right)\right)\\
&=W_{\alpha}\left(k\exp\left(-\int_b^a \frac{\breve L(t)}{t}\ \mathrm dt \right)\exp\left(-\int_a^x \frac{\breve L(t)}{t}\ \mathrm dt \right)\right)\\
&=W_{\alpha}\left(k_1\exp\left(-\int_a^x \frac{\breve L(t)}{t}\ \mathrm dt \right)\right),
\end{align*}
where $k_1$ is such that $k_1=k\exp\left(-\int_b^a \frac{\breve L(t)}{t}\ \mathrm dt \right)>0$.
\end{remark}

\begin{remark}
\label{ossnuova1}
If $\breve L$ is the $\alpha$--generalized reversed aging intensity function, with $\alpha<0$, of a non--negative and absolutely continuous random variable $X$, it satisfies conditions $\textit{(1)}$, $\textit{(2)}$, $\textit{(3)}$ of Theorem~\ref{T20}. In fact, from \eqref{eq4} we observe that $\breve L$ is non--negative for $x\in(0,+\infty)$. Moreover,
\begin{eqnarray*}
\lim_{x \to 0^+} \int_x^a \frac{\breve L(t)}{t}\ \mathrm dt &=& \lim_{x \to 0^+} \int_x^a \frac{\alpha F^{\alpha-1}(t)f(t)}{1-F^{\alpha}(t)}\ \mathrm dt \\
&=& \lim_{x \to 0^+} -\log\frac{1-F^{\alpha}(a)}{1-F^{\alpha}(x)}= +\infty,
\end{eqnarray*}
\begin{eqnarray*}
\lim_{x \to +\infty} \int_a^x \frac{\breve L(t)}{t}\ \mathrm dt &=& \lim_{x \to +\infty} \int_a^x \frac{\alpha F^{\alpha-1}(t)f(t)}{1-F^{\alpha}(t)}\ \mathrm dt \\
&=& \lim_{x \to +\infty} -\log\frac{1-F^{\alpha}(x)}{1-F^{\alpha}(a)}= +\infty.
\end{eqnarray*}
\end{remark}

\begin{remark}
If $\breve L$ is a function that satisfies conditions $\textit{(1)}$, $\textit{(2)}$, $\textit{(3)}$ of Theorem~\ref{T20} then it determines, for $\alpha=0$, a family of absolutely continuous distribution functions $F_k$ by the relationship
\begin{eqnarray}
\nonumber
F_k(x) &=& 1-W_{0}\left(k\exp\left(-\int_a^x \frac{\breve L(t)}{t}\ \mathrm dt \right)\right) \\
													 &=& \exp\left[-k\, \exp\left(-\int\limits_{a}\limits^{x} \frac{\breve{L}(t)}{t} \mathrm{d}t\right)\right], \ \ \ \ x\in(0, +\infty), 
\end{eqnarray}
for varying the parameter $k\in(0,+\infty)$ and it is the $0$--generalized reversed aging intensity function (i.e., the reversed aging intensity function) for those distribution functions. This follows from corollary 4 of Szymkowiak (2018a) noting that $L_{X}\left(\frac{1}{x}\right)=\breve L_{\frac{1}{X}}(x)$.
\end{remark}

In the following theorem we show that, for $\alpha>0$, the distribution function of a non--negative and absolutely continuous random variable is defined by the $\alpha$--generalized reversed aging intensity function and that, under some conditions, a function can be considerated as the $\alpha$--generalized reversed aging intensity function for a unique random variable.

\begin{theorem}
\label{T21}
Let $X$ be a non--negative and absolutely continuous random variable with cdf $F$ and let $\breve L_{\alpha}$ be its $\alpha$--generalized reversed aging intensity function with $\alpha>0$. Then $F$ and $\breve L_{\alpha}$ are related, for all $a\in(0,+\infty)$, by the relationship
\begin{equation}
\label{eq2}
F(x)=\left[1-\exp\left(-\int_0^{x} \frac{\breve L_{\alpha}(t)}{t}\ \mathrm dt \right)\right]^{\frac{1}{\alpha}},\ \ \ x\in(0,+\infty).
\end{equation}
Moreover, a function $\breve L$ defined on $(0,+\infty)$ and satisfying, for $a\in(0,+\infty)$, the following conditions:
\begin{itemize}
\item[(1)] $0\leq \breve L(x)+\infty,$ for all $x \in (0,+\infty)$;
\item[(2)] $\lim_{x \to 0^+} \int_x^a \frac{\breve L(t)}{t}\ \mathrm dt < +\infty$;
\item[(3)] $\lim_{x \to +\infty} \int_a^x \frac{\breve L(t)}{t}\ \mathrm dt = +\infty$;
\end{itemize}
determines, for $\alpha>0$, a unique absolutely continuous distribution function $F$ by the relationship
\begin{eqnarray}
\nonumber
F(x) &=& 1-W_{\alpha}\left(\frac{1}{\alpha}\exp\left(-\int_0^{x} \frac{\breve L(t)}{t}\ \mathrm dt \right)\right) \\
													 &=& \left[1-\exp\left(-\int_0^{x} \frac{\breve L(t)}{t}\ \mathrm dt \right)\right]^{\frac{1}{\alpha}}, \ \ \ x\in(0,+\infty),
\end{eqnarray}
 and it is $\alpha$--generalized reversed aging intensity function for that distribution function.
\end{theorem}

\begin{proof}
Fix distribution function $F$ with respective density function $f$, and put $\alpha>0$. From the definition of $\breve L_{\alpha}$ it is possible to obtain
$$\frac{\breve L_{\alpha}(t)}{t}=\frac{\alpha F^{\alpha-1}(t)f(t)}{1-F^{\alpha}(t)}, \mbox{  } t\in(0,+\infty).$$
By integrating both members between $0$ and $x$, we get
\begin{eqnarray*}
\int_0^{x} \frac{\breve L_{\alpha}(t)}{t}\ \mathrm dt &=& \int_0^{x} \frac{\alpha F^{\alpha-1}(t)f(t)}{1-F^{\alpha}(t)}\ \mathrm dt \\
													 &=& -\log(1-F^{\alpha}(x)),
\end{eqnarray*}
therefore
$$1-F^{\alpha}(x)=\exp\left(-\int_0^{x} \frac{\breve L_{\alpha}(t)}{t}\ \mathrm dt \right),$$
and so we get \eqref{eq2}.

Let $\breve L$ be a function defined on $(0,+\infty)$ and satisfying, for $a\in(0,+\infty)$, the conditions $\textit{(1)}$, $\textit{(2)}$, $\textit{(3)}$.
We show that $1-W_{\alpha}\left(\frac{1}{\alpha}\exp\left(-\int_0^{x} \frac{\breve L(t)}{t}\ \mathrm dt \right)\right)= F(x)$ defines a distribution function of a non--negative and absolutely continuous random variable.

In fact, from $\textit{(2)}$ it follows that $\lim_{x \to 0^+}  F(x)=0$, whereas from $\textit{(3)}$ we obtain that $\lim_{x \to +\infty} F(x)=1.$

Since $W_\alpha$ is increasing, $\alpha>0$ and the exponential function is increasing,  in order to show that it is an increasing function we have to prove that $-\int_0^{x} \frac{\breve L(t)}{t}\ \mathrm dt$ is a decreasing function in $x$ i.e., $\int_0^{x} \frac{\breve L(t)}{t}\ \mathrm dt$ is increasing in $x$, but this is immediate because the integrand is non negative and as $x$ increases, the integration interval widens.

Since $W_\alpha$, the exponential function, the multiplication for a scalar and the indefinite integral $x\mapsto\int_0^{x} \frac{\breve L(t)}{t}\ \mathrm dt$ are continuous functions, we have a continuous function. In order to obtain the absolute continuity of $F$, it suffices to observe that the derivative
\begin{equation*}
F'(x)=-\frac{1}{\alpha}\left[1-\exp\left(-\int_0^x \frac{\breve L(t)}{t} \mathrm dt\right)\right]^{\frac{1}{\alpha}-1} \exp\left(-\int_0^x \frac{\breve L(t)}{t} \mathrm dt\right)\left(-\frac{\breve L(x)}{x} \right)
\end{equation*}
is non--negative in $x>0$.
Finally, $F$ and $\breve L$ are related by the same relationship found in the first part of the theorem and so $\breve L$ is $\alpha$--generalized reversed aging intensity function for that distribution function.
\end{proof}

\begin{remark}
\label{ossnuova2}
If $\breve L$ is the $\alpha$--generalized reversed aging intensity function, with $\alpha>0$, of a non--negative and absolutely continuous random variable $X$, it satisfies conditions $\textit{(1)}$, $\textit{(2)}$, $\textit{(3)}$ of Theorem~\ref{T21}. In fact, from \eqref{eq4} we observe that $\breve L$ is non--negative for $x\in(0,+\infty)$. Moreover,
\begin{eqnarray*}
\lim_{x \to 0^+} \int_x^a \frac{\breve L(t)}{t}\ \mathrm dt &=& \lim_{x \to 0^+} \int_x^a \frac{\alpha F^{\alpha-1}(t)f(t)}{1-F^{\alpha}(t)}\ \mathrm dt \\
&=& \lim_{x \to 0^+} -\log\frac{1-F^{\alpha}(a)}{1-F^{\alpha}(x)}< +\infty,
\end{eqnarray*}
\begin{eqnarray*}
\lim_{x \to +\infty} \int_a^x \frac{\breve L(t)}{t}\ \mathrm dt &=& \lim_{x \to +\infty} \int_a^x \frac{\alpha F^{\alpha-1}(t)f(t)}{1-F^{\alpha}(t)}\ \mathrm dt \\
&=& \lim_{x \to +\infty} -\log\frac{1-F^{\alpha}(x)}{1-F^{\alpha}(a)}= +\infty.
\end{eqnarray*}
\end{remark}

In a concrete situation, if we have data it is possible to obtain an estimation of both distribution function and $\alpha$--generalized reversed aging intensity functions. So it could happen that the shape of an $\alpha$--generalized reversed aging intensity function is easier to recognize than that of the distribution function.

\section{Examples of characterization}

\begin{definition} 
We say that a random variable $X$ follows an inverse two-parameter Weibull distribution (see Murthy, Xie and Jiang 2004) if for $x\in(0,+\infty)$ and $\beta,\lambda>0$ the distribution function is expressed as
\begin{equation}
\label{eq5}
F(x)=\exp\left(-\frac{\lambda}{x^\beta}\right).
\end{equation}
In that case we write $X\sim invW2(\beta,\lambda)$.
\end{definition}

From the cdf \eqref{eq5} it is possible to obtain other characteristics of the distribution. In particular, for $x\in(0,+\infty)$, the pdf is
$$f(x)=\frac{\lambda\beta}{x^{\beta+1}}\exp\left(-\frac{\lambda}{x^\beta}\right),$$
the reversed hazard rate function is
$$\breve r(x)=\frac{f(x)}{F(x)}=\frac{\lambda\beta}{x^{\beta+1}},$$
and the $\alpha$--generalized reversed aging intensity function, for $\alpha\ne0$, is
\begin{eqnarray}
\nonumber
\breve L_\alpha(x) &=& \frac{\alpha x\frac{\lambda\beta}{x^{\beta+1}}\exp\left(-\frac{\lambda}{x^\beta}\right)\exp\left(-\frac{\lambda(\alpha-1)}{x^\beta}\right)}{1-\exp\left(-\frac{\lambda\alpha}{x^\beta}\right)} \\
\label{eq6}
													 &=& \frac{\alpha\beta\lambda}{x^\beta}\frac{\exp\left(-\frac{\lambda\alpha}{x^\beta}\right)}{1-\exp\left(-\frac{\lambda\alpha}{x^\beta}\right)}.
\end{eqnarray}

Let $\alpha<0$. By remark~\ref{ossnuova1} we know that \eqref{eq6} satisfies the hypothesis $\textit{(1)}, \textit{(2)},\textit{(3)}$ of Theorem~\ref{T20}. So we can apply the theorem by determining the quantity
\begin{eqnarray*}
F_k(x) &=& \left[1-k\alpha\exp\left(-\int_a^x \frac{\alpha\beta\lambda}{t^{\beta+1}}\frac{\exp\left(-\frac{\lambda\alpha}{t^\beta}\right)}{1-\exp\left(-\frac{\lambda\alpha}{t^\beta}\right)}\ \mathrm dt \right)\right]^{\frac{1}{\alpha}} \\
													 &=& \left[1-k\alpha\frac{\exp\left(-\frac{\lambda\alpha}{x^\beta}\right)-1}{\exp\left(-\frac{\lambda\alpha}{a^\beta}\right)-1}\right]^{\frac{1}{\alpha}}.
\end{eqnarray*}
\begin{corollary}
\label{cor1}
If a random variable $X$ has $\alpha$--generalized reversed aging intensity function, $\alpha<0$, $\breve L_{\alpha}(x)=\frac{\alpha\beta\lambda}{x^\beta}\frac{\exp\left(-\frac{\lambda\alpha}{x^\beta}\right)}{1-\exp\left(-\frac{\lambda\alpha}{x^\beta}\right)}$, a.e. $x\in(0,+\infty)$, with $\beta,\lambda>0$, then the distribution function of $X$ is expressed as
\begin{equation}
F(x)=\left[1-\gamma\alpha\left(\exp\left(-\frac{\lambda\alpha}{x^\beta}\right)-1\right)\right]^{\frac{1}{\alpha}},\mbox{  } x\in(0,+\infty)
\end{equation}
for $\gamma\in(0,+\infty)$.
\end{corollary}
\begin{remark}
If $\gamma=-\frac{1}{\alpha}$, the distribution function of Corollary \ref{cor1} is the distribution function of an inverse two-parameter Weibull distribution, $invW2(\beta,\lambda)$.
\end{remark}

Let $\alpha>0$. By remark~\ref{ossnuova2} we know that \eqref{eq6} satisfies the hypothesis $\textit{(1)}, \textit{(2)},\textit{(3)}$ of Theorem~\ref{T21}. So we can apply the theorem by determining the quantity
\begin{eqnarray*}
F(x) &=& \left[1-\exp\left(-\int_0^x \frac{\alpha\beta\lambda}{t^{\beta+1}}\frac{\exp\left(-\frac{\lambda\alpha}{t^\beta}\right)}{1-\exp\left(-\frac{\lambda\alpha}{t^\beta}\right)}\ \mathrm dt \right)\right]^{\frac{1}{\alpha}} \\
													 &=& \left[1-\left(1-\exp\left(-\frac{\lambda\alpha}{x^\beta}\right)\right)\right]^{\frac{1}{\alpha}}=\exp\left(-\frac{\lambda}{x^\beta}\right).
\end{eqnarray*}
\begin{corollary}
If a random variable $X$ has $\alpha$--generalized reversed aging intensity function, $\alpha>0$, $\breve L_{\alpha}(x)=\frac{\alpha\beta\lambda}{x^\beta}\frac{\exp\left(-\frac{\lambda\alpha}{x^\beta}\right)}{1-\exp\left(-\frac{\lambda\alpha}{x^\beta}\right)}$, a.e. $x\in(0,+\infty)$, with $\beta,\lambda>0$, then $X$ follows an inverse two-parameter Weibull distribution, $X\sim invW2(\beta,\lambda)$.
\end{corollary}

Let us consider some examples of polynomial $\alpha$--generalized reversed aging intensity functions.
\begin{example}
\label{example1}
Let us consider $\breve{L}_{\alpha}(x)=A>0$, for $x>0$. It could be a constant $\alpha$--generalized reversed aging intensity function for $\alpha\leq0$, in fact for $\alpha>0$ it does not satisfy the hypothesis of Theorem~\ref{T21}.

For $\alpha=0$ it determines a family of inverse two-parameter Weibull distributions by the relationship (see Szymkowiak (2018a))
\begin{equation}
F_k(x)=\exp\left[-k\left(\frac{1}{x}\right)^A\right],\ \ \ \ x\in(0, +\infty),
\end{equation}
where $k$ is a non--negative parameter.

For $\alpha<0$, it determines a family of continuous distributions by the relationship
\begin{equation}
F_k(x)=\left[1+k\left(\frac{1}{x}\right)^A\right]^{\frac{1}{\alpha}},\ \ \ \ x\in(0, +\infty)
\end{equation}
where $k$ is a non--negative parameter. 
\end{example}

\begin{example}
\label{example2}
Let us consider $\breve{L}_{\alpha}(x)=A+Bx$, for $x>0$ where $A,B>0$. It could be a linear $\alpha$--generalized reversed aging intensity function for $\alpha\leq0$, in fact for $\alpha>0$ it does not satisfy the hypothesis of Theorem~\ref{T21}.

For $\alpha=0$, it determines a family of continuous distributions by the relationship
\begin{equation}
F_k(x)=\exp\left[-k\left(\frac{1}{x}\right)^A\exp(-B\,x)\right],\ \ \ \ x\in(0, +\infty)
\end{equation}
where $k$ is a non--negative parameter.

For $\alpha<0$, it determines a family of continuous distributions by the relationship
\begin{equation}
F_k(x)=\left[1+k\left(\frac{1}{x}\right)^A\exp(-B\, x)\right]^{\frac{1}{\alpha}},\ \ \ \ x\in(0, +\infty)
\end{equation}
where $k$ is a non--negative parameter.
\end{example}

\begin{example}
Let us consider $\breve{L}_{\alpha}(x)=Bx$, for $x>0$, where $B>0$. It could be a linear $\alpha$--generalized reversed aging intensity function for $\alpha>0$, in fact for $\alpha>0$ it satisfies the hypothesis of Theorem~\ref{T21}. It determines a unique continuous distribution function by the relationship
\begin{equation}
F(x)=\left[1-\exp(-Bx)\right]^{\frac{1}{\alpha}},\ \ \ \ x\in(0, +\infty),
\end{equation}
i.e., an exponentiated exponential distribution (see Gupta and Kundu, 2001). We note that for $\alpha=1$ this is the distribution function of an exponential random variable with parameter $B$. So if $X$ has $1$--generalized reversed aging intensity function $\breve{L}_{1}(x)=Bx$, for $x>0$ and $B>0$ then $X\sim Exp(B)$.
\end{example}

\section{$\alpha$--generalized reversed aging intensity orders}\label{S5}
In this section we introduce and study the family of the $\alpha$--generalized reversed aging intensity orders. In the following, we use the notation $L_{\alpha,X}$ to indicate the $\alpha$--generalized aging intensity function of the random variable $X$ and $\breve L_{\alpha,X}$ to indicate the $\alpha$--generalized reversed aging intensity function of the random variable $X$.

In the next proposition we show a useful relationship between $L_{\alpha,X}$ and $\breve L_{\alpha,\frac{1}{X}}$.

\begin{proposition}
\label{pr}
Let $X$ be a non--negative and absolutely continuous random variable and let $\frac{1}{X}$ be its inverse. Then the following equality holds
\begin{equation}
L_{\alpha,X}\left(\frac{1}{x}\right)=\breve L_{\alpha,\frac{1}{X}}(x), \mbox{  } x\in(0,+\infty).
\end{equation}
\end{proposition}

\begin{proof}
We obtain an expression for the distribution function and the density function of the random variable $\frac{1}{X}$ through $X$, for $x>0$ we have
\begin{align*}
&F_{\frac{1}{X}}(x)=\mathbb P\left(\frac{1}{X}\leq x\right)=\mathbb P\left(X\ge \frac{1}{x}\right)=1-F_X\left(\frac{1}{x}\right),\\
&f_{\frac{1}{X}}(x)=\frac{1}{x^2}f_X\left(\frac{1}{x}\right).
\end{align*}
If $\alpha=0$ we have, for $x>0$,
\begin{align*}
\breve L_{0,\frac{1}{X}}(x)&=\breve L_{\frac{1}{X}}(x)=\frac{-xf_{\frac{1}{X}}(x)}{F_{\frac{1}{X}}(x)\log F_{\frac{1}{X}}(x)}\\
&=\frac{-\frac{1}{x}f_X\left(\frac{1}{x}\right)}{(1-F_X\left(\frac{1}{x}\right))\log(1-F_X\left(\frac{1}{x}\right))}=L_X\left(\frac{1}{x}\right)=L_{0,X}\left(\frac{1}{x}\right).
\end{align*}
If $\alpha\ne0$ we have, for $x>0$,
\begin{align*}
\breve L_{\alpha,\frac{1}{X}}(x)&=\frac{\alpha x(F_{\frac{1}{X}}(x))^{\alpha-1}f_{\frac{1}{X}}(x)}{1-(F_{\frac{1}{X}}(x))^{\alpha}}\\
&=\frac{\alpha\frac{1}{x}\left(1-F_X\left(\frac{1}{x}\right)\right)^{\alpha-1}f_X\left(\frac{1}{x}\right)}{1-(1-F_X\left(\frac{1}{x}\right))^{\alpha}}=L_{\alpha,X}\left(\frac{1}{x}\right).
\end{align*}
\end{proof}

\begin{definition}
Let $X$ and $Y$ be non--negative and absolutely continuous random variables and let $\alpha$ be a real number. We say that $X$ is smaller than $Y$ in the $\alpha$--generalized reversed aging intensity order, $X\leq_{\alpha RAI}Y$, if and only if $\breve L_{\alpha,X}(x)\leq \breve L_{\alpha,Y}(x)$, $\forall x\in(0,+\infty)$.
\end{definition}

In the next lemma we show a relationship between the $\alpha RAI$ order and the $\alpha AI$ order. We recall that $X\leq_{\alpha AI}Y$ if and only if $L_{\alpha,X}(x)\ge L_{\alpha,Y}(x)$, $\forall x\in(0,+\infty)$.

\begin{lemma}
\label{lemmaR}
Let $X$ and $Y$ be non--negative and absolutely continuous random variables and let $\alpha$ be a real number. We have $X\leq_{\alpha RAI}Y$ if and only if $\frac{1}{X}\ge_{\alpha AI}\frac{1}{Y}$.
\end{lemma}

\begin{proof}
We have $X\leq_{\alpha RAI}Y$ if and only if $\breve L_{\alpha,X}(x)\leq \breve L_{\alpha,Y}(x)$, $\forall x\in(0,+\infty)$. By proposition~\ref{pr} this is equivalent to $L_{\alpha,\frac{1}{X}}\left(\frac{1}{x}\right)\leq L_{\alpha,\frac{1}{Y}}\left(\frac{1}{x}\right)$, $\forall x\in(0,+\infty)$, i.e. $\frac{1}{X}\ge_{\alpha AI}\frac{1}{Y}$.
\end{proof}

\begin{remark}
\label{rem2}
For particular choices of the real number $\alpha$ we find some relationship with other stochastic orders. Obviously, the reversed aging intensity order coincides with the $0$--generalized reversed aging intensity order. For $\alpha=1$ we have showed in remark~\ref{rem} that $\breve L_{1,X}(x)=xr_X(x),$
so we get a relationship with the hazard rate order. In fact,
\begin{align}
\nonumber
X\leq_{hr}Y &\Leftrightarrow r_X(x)\ge r_Y(x), \forall x>0 \Leftrightarrow xr_X(x)\ge xr_Y(x), \forall x>0\\
&\Leftrightarrow \breve L_{1,X}(x)\ge \breve L_{1,Y}(x), \forall x>0 \Leftrightarrow X\ge_{1RAI}Y.
\end{align}
For $\alpha=-1$ we have showed in remark~\ref{rem} that $\breve L_{-1,X}(x)=xLOR_X(x)=L_{-1,X}(x),$
so we get a relationship with the log-odds rate order. In fact,
\begin{align}
\nonumber
X\leq_{LOR}Y &\Leftrightarrow LOR_X(x)\ge LOR_Y(x), \forall x>0 \Leftrightarrow xLOR_X(x)\ge xLOR_Y(x), \forall x>0\\
&\Leftrightarrow \breve L_{-1,X}(x)\ge \breve L_{-1,Y}(x), \forall x>0 \Leftrightarrow X\ge_{-1RAI}Y.
\end{align}
Moreover we have $X\ge_{-1RAI}Y \Leftrightarrow X\leq_{-1AI}Y$ so they are dual relations.
For $\alpha=n\in\mathbb N$ we have showed in remark~\ref{rem} that 
$$\breve L_{n,X}(x)=\frac{n x(F_X(x))^{n-1}f_X(x)}{1-(F_X(x))^{n}}=xr_{X_{(n)}}(x),$$
so there is a connection with the largest order statistic and the hazard rate order. In fact
\begin{align}
\nonumber
X_{(n)}\leq_{hr}Y_{(n)} &\Leftrightarrow r_{X_{(n)}}(x)\ge r_{Y_{(n)}}(x), \forall x>0 \Leftrightarrow xr_{X_{(n)}}(x)\ge xr_{Y_{(n)}}(x), \forall x>0\\
&\Leftrightarrow \breve L_{n,X}(x)\ge \breve L_{n,Y}(x), \forall x>0 \Leftrightarrow X\ge_{nRAI}Y.
\end{align}
\end{remark}

\begin{remark}
Lemma \ref{lemmaR} and Remark \ref{rem2} provide the following series of relations
\begin{equation*}
X\leq_{-1RAI}Y \Leftrightarrow \frac{1}{X}\leq_{-1RAI}\frac{1}{Y}\Leftrightarrow X\ge_{-1AI}Y\Leftrightarrow \frac{1}{X}\ge_{-1AI}\frac{1}{Y}.
\end{equation*}
\end{remark}

\begin{proposition}
\label{pp}
Let $X$ and $Y$ be non--negative and absolutely continuous random variables such that $X\leq_{st}Y$, i.e., $F_X(x)\ge F_Y(x)$ for all $x>0$.
\begin{itemize}
\item[(1)] If exists $\beta\in\mathbb R$ such that $X\leq_{\beta RAI} Y$ then for all $\alpha<\beta$ we have $X\leq_{\alpha RAI} Y$;
\item[(2)] If exists $\beta\in\mathbb R$ such that $X\ge_{\beta RAI} Y$ then for all $\alpha>\beta$ we have $X\ge_{\alpha RAI} Y$.
\end{itemize}
\end{proposition}

\begin{proof}
$(1)$. From $X\leq_{\beta RAI} Y$ and lemma~\ref{lemmaR} we have $\frac{1}{X}\ge_{\beta AI}\frac{1}{Y}$. Moreover from $X\leq_{st}Y$ we get $\frac{1}{X}\ge_{st}\frac{1}{Y}$ so by proposition $4$ of Szymkowiak (2018b) we obtain that $\forall\alpha<\beta$  $\frac{1}{X}\ge_{\alpha AI}\frac{1}{Y}$, i.e., $X\leq_{\alpha RAI}Y$.

The proof of part $(2)$ is analogous.
\end{proof}

\begin{proposition}
\label{pp2}
Let $X$ and $Y$ be non--negative and absolutely continuous random variables.
\begin{itemize}
\item[(1)] If exists $\beta\in\mathbb R$ such that for all $\alpha<\beta$ we have $X\ge_{\alpha RAI} Y$ then $X\ge_{rh} Y$, i.e., $\breve r_X(x)\ge\breve r_Y(x)$ for all $x>0$;
\item[(2)] If exists $\beta\in\mathbb R$ such that for all $\alpha>\beta$ we have $X\leq_{\alpha RAI} Y$ then $X\ge_{st} Y$.
\end{itemize}
\end{proposition}

\begin{proof}
$(1)$. From $X\ge_{\alpha RAI} Y$ and lemma~\ref{lemmaR} we have $\frac{1}{X}\leq_{\alpha AI}\frac{1}{Y}$, $\forall \alpha<\beta$. So with use of proposition $5$ of Szymkowiak (2018b) we obtain $\frac{1}{X}\leq_{hr}\frac{1}{Y}$, i.e. $X\ge_{rh}Y$.

The proof of part $(2)$ is analogous.
\end{proof}

\begin{corollary}
Let $X$ and $Y$ be non--negative and absolutely continuous random variables.
\begin{itemize}
\item[(1)] $X\leq_{st}Y$ and $X\ge_{LOR}Y\Rightarrow$ $X\leq_{\alpha RAI} Y$ for all $\alpha\in(-\infty,-1)$;
\item[(2)] $X\leq_{st}Y$ and $X\leq_{LOR}Y\Rightarrow$ $X\ge_{\alpha RAI} Y$ for all $\alpha\in(-1,+\infty)$.
\end{itemize}
\end{corollary}

\begin{proof}
$(1)$. We have $X\ge_{LOR}Y\Leftrightarrow X\leq_{-1RAI} Y$ so the proof is completed with use of proposition~\ref{pp}.

The proof of part $(2)$ is analogous.
\end{proof}

\section{Application of $\alpha$--generalized reversed aging intensity function in data analysis}

Very often it is really a difficult task to recognize the lifetime data distribution analyzing only the shapes of their density and distribution function estimators. But sometimes, the corresponding $\alpha$--generalized reversed aging intensity function for a properly chosen $\alpha$
can have a relatively simple form, and it can be easily recognized with use of the respective reversed aging intensity estimate.

For some distribution $F$ with support $x\in(0, +\infty)$, we obtain a natural estimator of the $\alpha$--generalized reversed aging intensity function

\begin{equation}\label{f1}
\widehat{\breve{L}}_{\alpha}(x)
=
\left\{ \begin{array}{lll}
\frac{\alpha\, x\,\widehat{f}(x) [\widehat{F}(x)]^{\alpha-1} }{1-[\widehat{F}(x)]^{\alpha}} & \textrm{for}&  x>0,\ \ \alpha\neq 0\\
-\frac{x\, \widehat{f}(x)}{\widehat{F}(x)\ln[\widehat{F}(x)]}& \textrm{for}&x>0,\ \ \alpha= 0,
\end{array} \right.
\end{equation}
where $\widehat{f}$ denotes a nonparametric estimate 
of the unknown density function $f$ and $\widehat{F}(x)= \int_0^x \widehat{f}(t)\mathrm{d}t$ represents the corresponding distribution function estimate. The proposed estimation of the aging intensity function is possible if we assume that data follow an absolutely continuous distribution with support $(0, +\infty)$ and if the nonparametric estimate of its density function exists.  Moreover, larger sample sizes generally lead to increased precision of estimation. We perform our study for both the generated and the real data.

\subsection{Analysis of $\alpha$--generalized reversed aging intensity function through generated data}
In the following example we consider an application of the estimator (\ref{f1}) for $\alpha=-1$ to verify the hypothesis that some simulated data come
from the family of inverse log-logistic distributions.

\begin{example}\label{ex1}
{\upshape Our goal is to check if a member of the inverse log-logistic distributions $invLLog(\gamma, \lambda)$ 
with the distribution function given by
\begin{equation}\label{f2}
F_{\gamma, \lambda}(x) =\left[1+\left(\frac{\lambda}{x}\right)^{\gamma}\right]^{-1},\ \ \ \ x\in(0, +\infty),
\end{equation}
for some unknown positive parameters of the shape $\gamma$ and the scale $\lambda$, is the parent distribution of a random sample $X_1,\ldots,X_N$.

From presented in Section 4, Example \ref{example1} we know that for distribution function (\ref{f2}), its $-1$--generalized reversed aging intensity function is constant and 
equal to $\breve{L}_{-1}(x)=\gamma$.
So, we check if the respective reversed aging intensity estimator (\ref{f1}) is indeed an accurate approximation of a constant function.

Therefore, we use the following procedure to obtain $N$ independent random variables $X_1,\ldots,X_N$ with $invLLog(\gamma, \lambda)$ lifetime distribution.
First, we generate standard uniform random variables $U_1,\ldots,U_N$ using function \texttt{random} of {\it MATLAB}. 
Then, applying the inverse transform technique with $F_{\gamma, \lambda}(x)= \left[1+\left(\frac{\lambda}{x}\right)^{\gamma}\right]^{-1}$,
we get $Y_i= F^{-1}_{\gamma, \lambda}(1-U_i)=\lambda \left(\frac{1}{1-U_i}-1\right)^{-\frac{1}{\gamma}}$, $i=1, \ldots , N$, with the inverse log-logistic distribution
$invLLog(\gamma, \lambda)$.
In this way, applying the function \texttt{random} with the \texttt{seed}$=88$, we generate  $N=1000$ independent inverse log-logistic
random variables with the shape parameter $\gamma=4$, and the scale parameter $\lambda=0.5$.

To calculate the reversed aging intensity estimator (\ref{f1}), we apply a kernel density estimator (see Bowman and Azzalini 1997), i.e., given in {\it MATLAB} \texttt{ksdensity} function,
\begin{equation}\label{f3}
\widehat{f}(x)=\frac{1}{N\, h}\sum_{j=1}^{N}K\left(\frac{x-X_j}{h}\right),
\end{equation}
with a chosen normal kernel smoothing function and a selected bandwidth $h=0.05$.
Then, the kernel estimator of the distribution function is equal to
$$
\widehat{F}(x)=\frac{1}{N}\sum_{j=1}^{N}I\left(\frac{x-X_j}{h}\right) ,
$$
where $I(x)=\int_{-\infty}^x K(t)\mathrm{d}t$. 
The obtained $-1$--generalized reversed aging intensity function estimate (\ref{f1}) is equal to

\vspace{-0,5cm}
\begin{equation}\label{f4}
\widehat{\breve{L}}_{-1}(x) =\frac{ x\,\widehat{f}(x)}{\widehat{F}(x)\left[1-\widehat{F}(x)\right]}=\frac{x\, \frac{1}{N\, h}\sum_{j=1}^{N}K\left(\frac{x-X_j}{h}\right)}{\frac{1}{N}\sum_{j=1}^{N}I\left(\frac{x-X_j}{h}\right)\left[1-\frac{1}{N}\sum_{j=1}^{N}I\left(\frac{x-X_j}{h}\right)\right]}.
\end{equation}
For our simulation data, the plot of the density estimator (\ref{f3}) is presented in Figure \ref{pdf_invLLog}.

\begin{figure}[h!]
\vspace{-7cm}
\begin{center}
\includegraphics[width=01\textwidth]{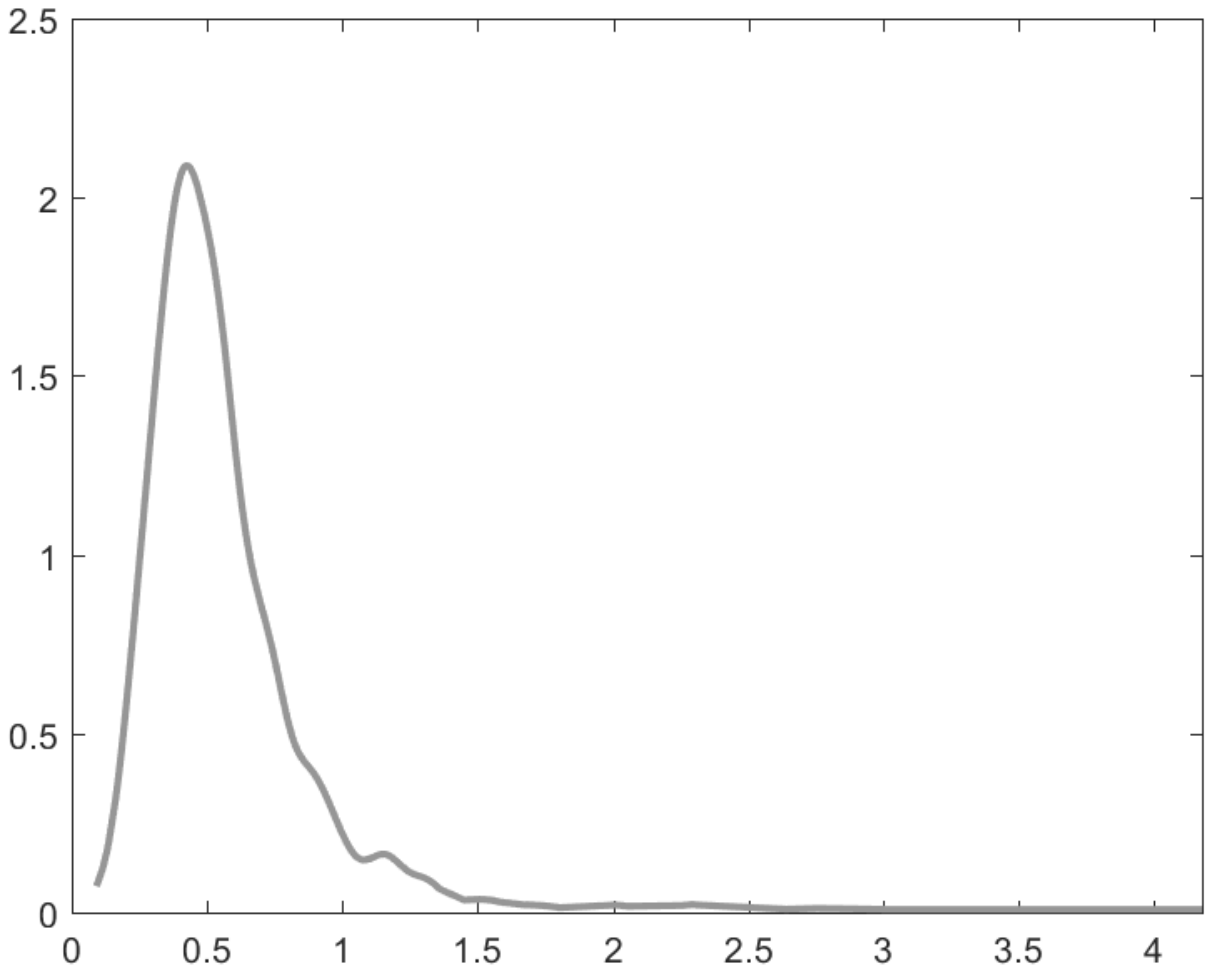}
\vspace{-8,5cm}
\caption{Density estimator $\widehat{f}(x)$ for the data from Example \ref{ex1}}\label{pdf_invLLog}
\end{center}
\vspace{-0,5cm}
\end{figure}

Analyzing the plot, it is not easy to decide if the density function belongs to the inverse log-logistic family. 
But, we can notice that the plot of respective estimator (\ref{f4}) of $-1$--generalized reversed aging intensity function $\widehat{\breve{L}}_{-1}(x)$ (see Figure \ref{RAI_invLLog}),
oscillates around a constant function, 
especially after removing few outlying values at the right-end. 
This gives us the motivation to accept our hypothesis that an inverse log-logistic distribution is the parent distribution of the generated sample.

\begin{figure}[h!]
\vspace{-7cm}
\begin{center}
\includegraphics[width=1\textwidth]{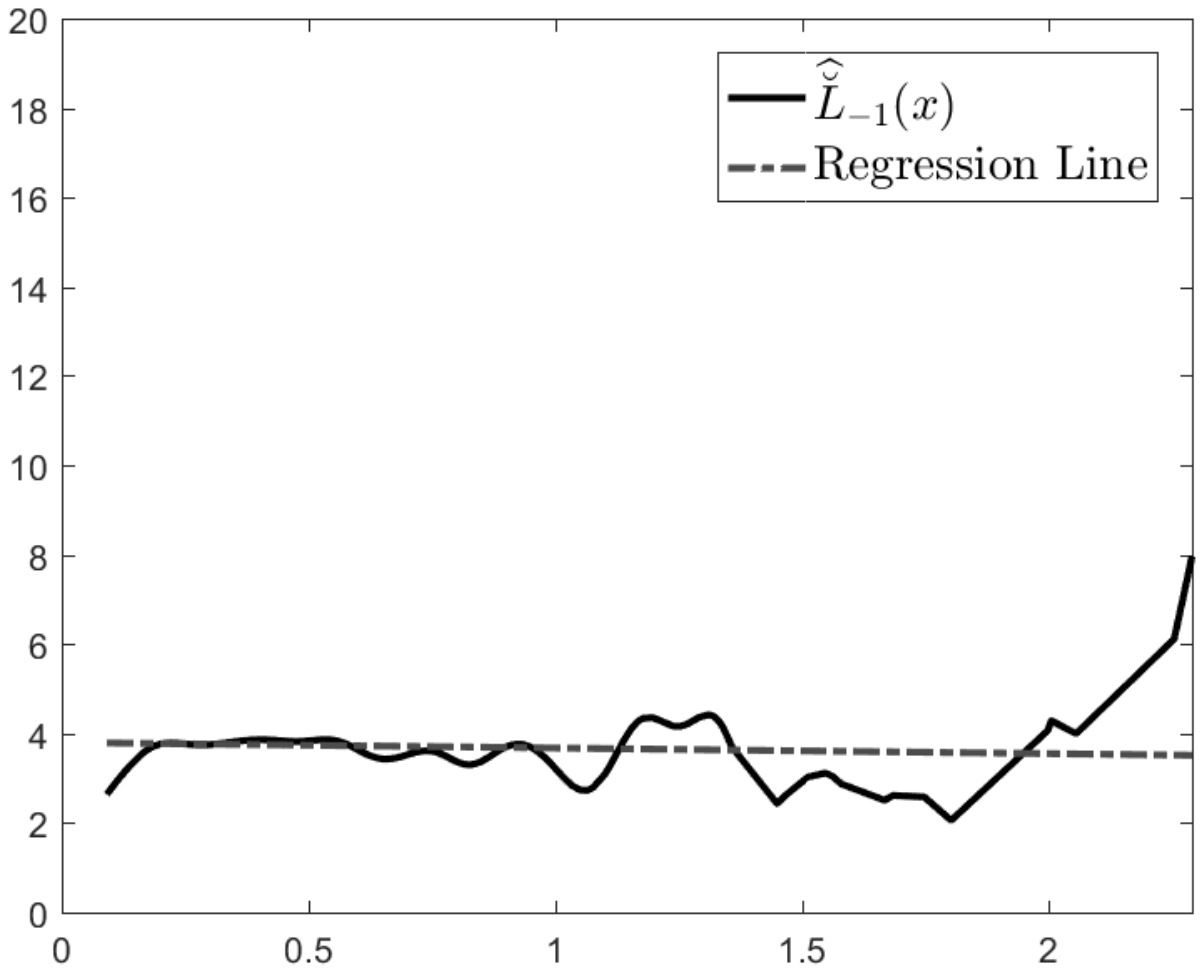}
\vspace{-8,5cm}
\caption{$\widehat{\breve{L}}_{0}(x)$
and adjusted regression line for the data from Example \ref{ex1}}\label{RAI_invLLog}
\end{center}
\vspace{-0,5cm}
\end{figure}

To justify our intuitive decision, we propose to carry out the following more formal statistical procedure.
First, we calculate the least squares estimate of the intercept which for our data equals to $\widehat{\gamma}=3.7990$. 
Next, we put it into the log-likelihood function, and determine maximum likelihood estimator (MLE) of parameter $\lambda$ maximizing it.
The problem resolves into finding the solution to the equation
$$\sum_{i=1}^N\frac{1}{\left(\frac{x_i}{\lambda}\right)^{\widehat{\gamma}}+1}=\frac{N}{2}.$$
As the result we obtain $\widehat{\lambda}=0.4957$. Note that the estimators 
$\widehat{\gamma}$ and $\widehat{\lambda}$
based on the empirical $-1$--generalized reversed aging intensity are quite precise (cf. Table \ref{Tab2}).
{\small
\begin{table}[h!]
\vspace{-0,4cm}
\caption{Parameters of $invLLog(\gamma, \lambda)$}\label{Tab2}
\vspace{-0,2cm}
\begin{center}
\begin{tabular}{|c|c|c|} \hline
&$\gamma$&$\lambda$\\ \hline
Theoretical parameters &4&0.5\\ \hline
Estimators&3.7990&0.4957\\ \hline
\end{tabular}
\vspace{-0,5cm}
\end{center}
\end{table}}

Finally, by the chi-square goodness-of-fit test we check if the data really fit the inverse log-logistic distribution. For this purpose, we apply function \texttt{histogram}, available in \texttt{MATLAB} and group the data into $k=20$ classes of observations lying into intervals
$[x_j,x_{j+1})= [x_j, x_j+\Delta x)$, $j=1,\ldots,k$,
of length $\Delta x=0.21$. The classes, together with their empirical frequencies $N_j=N_j(X_1,\ldots,X_N)$ and theoretical frequencies based on the
inverse log-logistic distribution with parameters replaced by the estimators
$n_j= N \left[F_{\widehat{\gamma},\widehat{\lambda}}(x_{j+1})- 
F_{\widehat{\gamma},\widehat{\lambda}}(x_j)\right]$, are presented in Table \ref{Tab3}.
{\small
\begin{table}[h!]
\vspace{-0,3cm}
\caption{Grouped data and respective values of empirical and theoretical frequency}\label{Tab3}
\vspace{-0,2cm}
\begin{center}
\begin{tabular}{|r|c|r|r|}\hline
class&$[x_j, x_{j+1})$ &$N_j$&$n_j$\\ \hline \hline
1&0.0000-0.2100&26&36.8543\\ \hline
2&0.2100-0.4200&322&310.6691\\ \hline
3&0.4200-0.6300&371&365.5653\\ \hline
4&0.6300-0.8400&150&168.0593\\ \hline
5&0.8400-1.0500&68&64.2263\\ \hline
6&1.0500-1.2600&29&26.5322\\ \hline
7&1.2600-1.4700&15&12.2550\\ \hline
8&1.4700-1.6800&5&6.2413\\ \hline
9&1.6800-1.8900&3&3.4409\\ \hline
10&1.8900-2.1000&3&2.0223\\ \hline
11&2.1000-2.3100&0&1.2522\\ \hline
12&2.3100-2.5200&1&0.8095
\\ \hline
13&2.5200-2.7300&2&0.5425\\ \hline
14&2.7300-2.9400&1&0.3749\\ \hline
15&2.9400-3.1500&0&0.2660\\ \hline
16&3.1500-3.3600&0&0.1931\\ \hline
17&3.3600-3.5700&0&0.1430\\ \hline
18&3.5700-3.7800&0&0.1078\\ \hline
19&3.7800-3.9900&0&0.0826\\ \hline
20&3.9900-4.200&1&0.0641\\ \hline
\end{tabular}
\vspace{-0,5cm}
\end{center}
\end{table}}

Furthermore, available in {\it MATLAB} function \texttt{chi2gof} determines the value of chi-square statistics $\chi^2=9.3209$ with $\nu= 7$
degrees of freedom (automatically joining together the last twelve classes with low frequencies) and determines the respective $p$-value, $p= 0.2304$.
It means that for a given significance level less than $0.2304$ we do not reject the hypothesis that the considered data follow the inverse log-logistic distribution.
}
\end{example}

\subsection{Analysis of $\alpha$--generalized reversed aging intensity through real data}
Next, we present an example of real data. Analyzing its estimated $\alpha$--generalized reversed aging intensity we could assume that the data follow the adequate distribution.

\begin{example}\label{E2}
{\upshape The real data (see Data Set 6.2 in Murthy et al. 2004) concern failure times of 20 components: 0.067 0.068 0.076 0.081 0.084 0.085 0.085 0.086 0.089 0.098 0.098 0.114 0.114 0.115 0.121 0.125 0.131 0.149 0.160 0.485.

For the given data, the plot of the normal kernel density estimator
(see Bowman and Azzalini 1997), obtained by {\it MATLAB} function \texttt{ksdensity} with a returned bandwidth $h=0.0147$, is presented in Figure \ref{pdf_invMW}. An analysis of the graph does not enable us to recognize the data distribution. 
\begin{figure}[h!]
\vspace{-7cm}
\begin{center}
\includegraphics[width=1\textwidth]{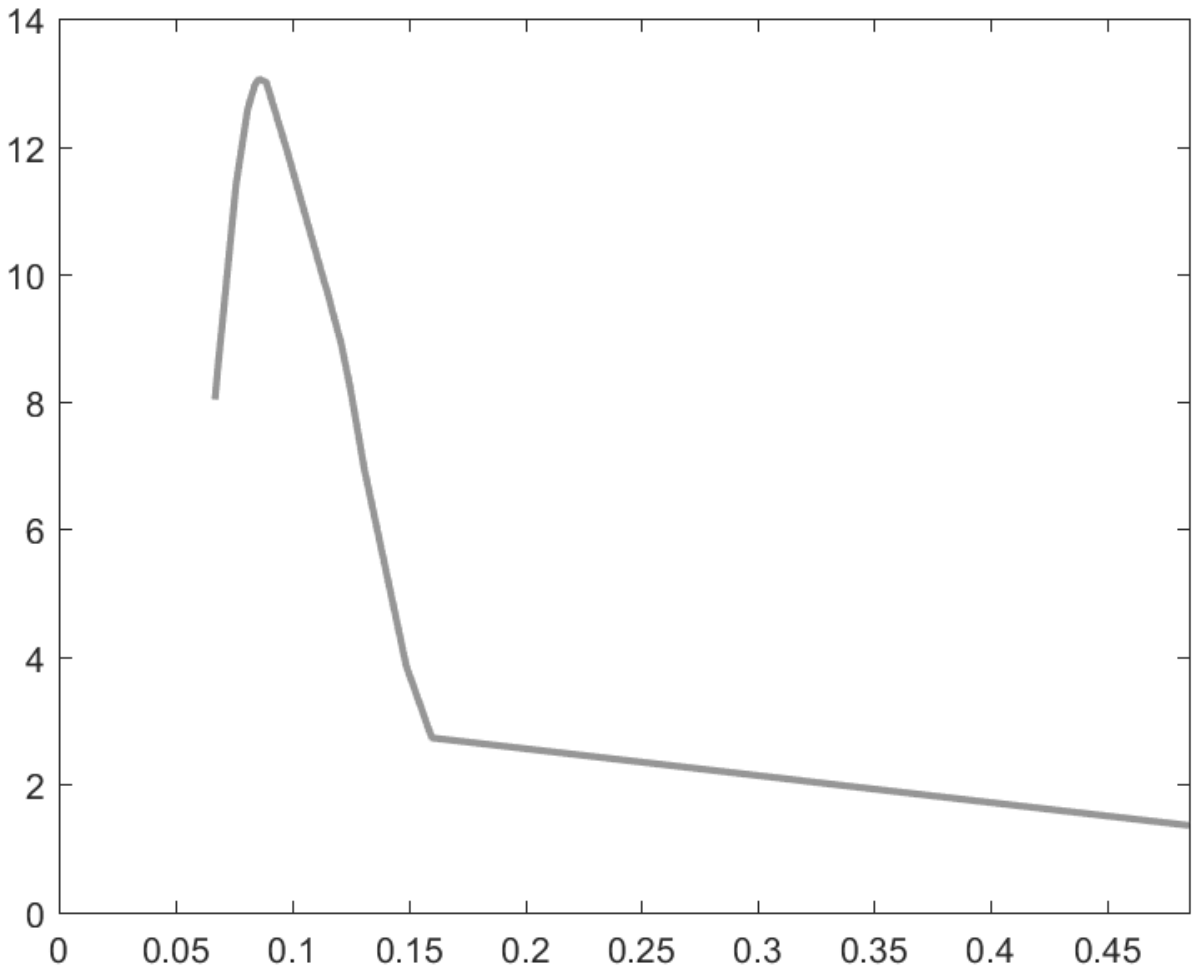}
\vspace{-8,5cm}
\caption{Kernel density estimator $\widehat{f}$ for the data from Example \ref{E2}}\label{pdf_invMW}
\end{center}
\vspace{-0,5cm}
\end{figure}
To identify the data distribution we propose to estimate $0$--generalized reversed aging intensity (see formula (\ref{f1}))
$$\widehat{\breve{L}}_{0} (x) =-\frac{x\, \widehat{f}(x)}{\widehat{F}(x)\ln[\widehat{F}(x)]}, \ \  x\in(0, +\infty).$$
\begin{figure}[h!]`	
\vspace{-7cm}
\begin{center}
\includegraphics[width=1\textwidth]{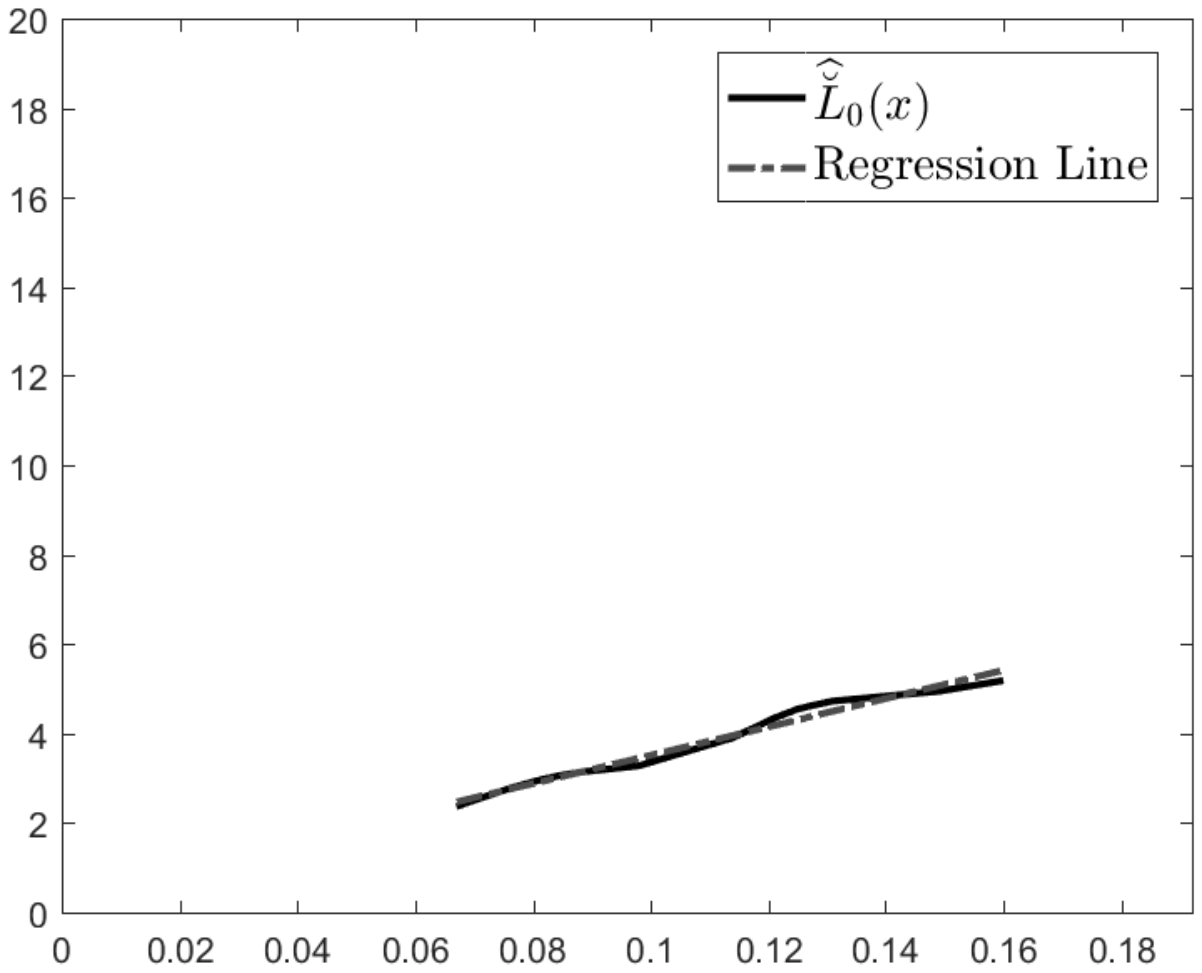}
\vspace{-8,5cm}
\caption{$\widehat{\breve{L}}_{0} (x)$,
and adjusted regression line for the data from Example \ref{E2}}\label{RAI_invMW}
\end{center}
\vspace{-0,5cm}
\end{figure}
The plot of the estimator $\widehat{\breve{L}}_0 (x)$ (see Figure \ref{RAI_invMW}) can be treated as oscillating around a linear function, especially after removing one outlying value at the right-end. This motivates us to state the hypothesis that data follow an inverse modified Weibull distribution (see Section 4, Example \ref{example2}) with distribution function
\begin{equation}\label{f5}
F_{\gamma, \lambda, \delta}(x) =\exp\left[-\left(\frac{\lambda}{x}\right)^\gamma \exp(-\delta\, x)\right],\ \ \ \ x\in(0, +\infty),
\end{equation}
and $0$--generalized reversed aging intensity function
$$\breve{L}_0(x)=\delta\, x+\gamma, \ \ \ \ x\in(0, +\infty) .$$
Moreover, we provide the following procedure. First, we determine the least squares estimates $\widehat{\gamma}=0.3441$  and $\widehat{\delta}=31.6785$ of the intercept and the slop of linear $\breve{L}_0$, respectively. Then we determine MLE of parameter $\lambda$
$$\widehat{\lambda}= \left(\frac{N}{\sum_{i=1}^{N}\frac{\exp(-\widehat{\delta}\, x_i)}{(x_i)^{\widehat{\gamma}}}}\right)^{\frac{1}{\widehat{\gamma}}}$$ which maximizes the likelihood function. Here we obtain $\widehat{\lambda}=549.9663$.

Then, to check if the data fit the inverse modified Weibull distribution we use adequate for small data the Kolmogorov-Smirnov goodness-of-fit test (avaliable in {\it MATLAB} function \texttt{kstest}), we determine statistics $K=0.1496$  and $p$-value of the test equal to $p=0.7072$. It means that for a given significance level less than $0.7072$ we do not reject the hypothesis that the considered data follow the inverse modified Weibull  distribution.

}
\end{example}

\section{Conclusion}

In this paper, a family of generalized reversed aging intensity functions was introduced and studied. In particular, it was showed that, using the generalized Pareto distribution to generalize the concept of reversed aging intensity function, for $\alpha>0$, the $\alpha$--generalized reversed aging intensity function characterizes a unique distribution function, while for $\alpha\leq0$, it determines a family of distribution functions. Moreover, $\alpha$--generalized reversed aging intensity orders were introduced and some relations with other stochastic orders were studied. Finally, analysis of $\alpha$--generalized reversed aging intensity through generated data and real one are given.

\section*{Acknowledgement}
Francesco Buono and Maria Longobardi are partially supported by the GNAMPA research group of INdAM (Istituto Nazionale di Alta Matematica) and MIUR-PRIN 2017, Project "Stochastic Models for Complex Systems" (No. 2017 JFFHSH). 

\noindent Magdalena Szymkowiak is partially supported by PUT under grant 0211/SBAD/0911.

%
%
%
%

\end{document}